\documentclass[a4paper]{article}
\usepackage{amsmath,amssymb,amscd,amsthm}
\usepackage{epic,eepic}
\usepackage[dvips]{graphics}
\usepackage{comment}

\numberwithin{equation}{section}

\DeclareFontEncoding{OT2}{}{}
\DeclareFontSubstitution{OT2}{wncyr}{m}{n}
\DeclareSymbolFont{cyss}{OT2}{wncyss}{m}{n}
\DeclareSymbolFont{cyr}{OT2}{wncyr}{m}{n}
\DeclareMathSymbol{\sh}{\mathbin}{cyss}{`x}

\newcommand{\Li}{\operatorname{Li}}
\newcommand{\reg}{\operatorname{reg}}

\newcommand{\C}{{\mathbf C}}
\newcommand{\D}{{\mathbf D}}
\newcommand{\Q}{{\mathbf Q}}
\newcommand{\bP}{{\mathbf P}}
\newcommand{\bunit}{{\mathbf I}}
\newcommand{\bnull}{{\mathbf 1}}

\newcommand{\fX}{{\mathfrak X}}
\newcommand{\cL}{{\mathcal L}}
\newcommand{\cM}{{\mathcal M}}
\newcommand{\cU}{{\mathcal U}}
\newcommand{\tcU}{{\widetilde{\mathcal U}}}
\newcommand{\hcL}{{\widehat{\mathcal L}}}

\newcommand{\tF}{\tilde{F}}
\newcommand{\tf}{\tilde{f}}

\newcommand{\ds}{\displaystyle}
\newtheorem{thm}{Theorem}

\newtheorem{prop}[thm]{Proposition}
\newtheorem{lem}[thm]{Lemma}

\title{Fundamental solutions of the Knizhnik-Zamolodchikov equation of one variable and the Riemann-Hilbert problem}
\author{OI, Shu\thanks{Department of Mathematics, College of Science, Rikkyo  University. \endgraf \hspace{1em}3-34-1, 
Nishi-Ikebukuro, Toshima-ku, Tokyo 171-8501, Japan. \endgraf \hspace{1em}e-mail: {\tt shu-oi@rikkyo.ac.jp}}
\and UENO, Kimio\thanks{Department of Mathematics, 
School of Fundamental Sciences and Engineering,\endgraf \hspace{1em} Faculty of Science and Engineering, 
Waseda University. \endgraf \hspace{1em}3-4-1, Okubo, Shinjuku-ku, 
Tokyo 169-8555, Japan. \endgraf \hspace{1em}e-mail: {\tt uenoki@waseda.jp}
}}
\date{}

\begin{document}


\insert\footins{\footnotesize 2010 {\it Mathematics Subject Classification.} Primary 34M50,11G55; 
Secondary 30E25,11M06,32G34;}
\maketitle

\begin{abstract}
In this article, we derive multiple polylogarithms from multiple zeta values 
by using a recursive Riemann-Hilbert problem of additive type. Furthermore we show that 
this Riemann-Hilbert problem is regarded as an inverse problem for the connection problem of the KZ equation 
of one variable, so that the fundamental solutions to the equation are derived from the Drinfel'd associator 
by using a Riemann-Hilbert problem of multiplicative type. 
These results say that the duality relation for the Drinfel'd associator can be interpreted as the solvability condition for this inverse problem.
\end{abstract}

\section{Introduction}
In \cite{OiU2}, we showed that the polylogarithms $\Li_k(z)$ are characterized by the inversion formula of polylogarithms 
\begin{equation}
\Li_k(z)+\sum_{j=1}^{k-1}\frac{(-1)^j\log^j z}{j!}\Li_{k-j}(z)+\Li_{2,\underbrace{\scriptstyle 1,\ldots,1}_{k-2}}(1-z)
=\zeta(k), \label{eq:inversion01}
\end{equation}
which is viewed as a recursive Riemann-Hilbert problem of additive type. Generalizing this scheme, 
we give a characterization of the multiple polylogarithms of one variable $\Li_{k_1,\ldots,k_r}(z)$.\\

The multiple polylogarithms of one variable are holomorphic functions on $|z|<1$ determined by the Taylor expansions
\begin{equation}\label{eq:1MPL}
\Li_{k_1,\ldots,k_r}(z)=\sum_{n_1>\cdots>n_r>0}\frac{z^{n_1}}{n_1^{k_1}\cdots n_r^{k_r}}
\end{equation}
where $r \ge 1, k_1,\ldots,k_r\ge 1$. 
By using iterated integrals \eqref{eq:iterated_integral}, they can be expressed as 
\begin{equation}
\Li_{k_1,\ldots,k_r}(z) = \int_0^z \Big(\frac{dz}{z}\Big)^{k_1-1} \frac{dz}{1-z} \cdots 
                        \Big(\frac{dz}{z}\Big)^{k_r-1} \frac{dz}{1-z},\vspace*{1mm}
\end{equation}
and can be continued onto $\bP^1\setminus\{0,1,\infty\}$ as many-valued analytic functions 
along the integral path, where $\bP^1$ denotes the Riemann sphere. 

The generating function \eqref{eq:fundamental0} of the multiple polylogarithms yields a fundamental solution 
of the Knizhnik-Zamolodchikov equation (the KZ equation, for short) of one variable. 
This is an ordinary differential equation
\begin{equation}\label{eq:1KZ}
\frac{dG}{dz}=\left(\frac{X_0}{z}+\frac{X_1}{1-z}\right)G
\end{equation}
defined on the moduli space
\begin{equation*}
\cM_{0,4}=\bP^1 \setminus\{0,1,\infty\},
\end{equation*}
where $X_0, X_1$ are generators of the free Lie algebra $\fX=\C\{X_0,X_1\}$, which is a Lie algebra derived 
from the lower central series of the fundamental group of $\cM_{0,4}$ \cite{I}. 
The 1-forms $\xi_0=\frac{dz}{z}$ and $\xi_1=\frac{dz}{1-z}$ are considered as dual variables of $X_0, X_1$ 
and generate a shuffle algebra. This algebra describes iterated integrals of the forms $\xi_0$ and $\xi_1$.

Moreover the connection matrix between the fundamental solution $\cL^{(0)}(z)$ of \eqref{eq:1KZ} normalized 
at $z=0$ (see Section \ref{sec:prep}, Proposition \ref{prop:KZE1}) and the fundamental solution $\cL^{(1)}(z)$ 
normalized at $z=1$ is given by the Drinfel'd associator $\varPhi_{\rm KZ}$; namely
\begin{equation}
    \left(\cL^{(1)}(z)\right)^{-1}\cL^{(0)}(z)=\varPhi_{\rm KZ}. \label{eq:connection01}
\end{equation}
The elements $\cL^{(0)},\, \cL^{(1)}$ and $\varPhi_{\rm KZ}$ are grouplike elements of 
$\widetilde{\cU}(\fX)=\C\langle\langle X_0,X_1 \rangle\rangle$, which denotes 
the non-commutative formal power series algebra of the variables $X_0, X_1$. 

The Drinfel'd associator $\varPhi_{\rm KZ}$ is expressed as the generating function \eqref{eq:DA} of 
the multiple zeta values
\begin{equation}
\zeta(k_1,k_2,\ldots,k_r)=\sum_{n_1>\cdots>n_r>0}\frac{1}{n_1^{k_1}\cdots n_r^{k_r}}.
\end{equation}
In \cite{OkU}, it was shown that the connection relation \eqref{eq:connection01} is equivalent to 
the system of the functional relations \eqref{eq:gif} among extended multiple polylogarithms
\begin{align}\label{eq:gif_ext}
              \sum_{uv=w}\Li(\tau(u);1-z)\Li(v;z)=\zeta(\reg^{10}(w)),
\end{align}
referred to as ``the generalized inversion formulas.'' 
Here $w$ denotes a word of $\xi_0$ and $\xi_1$. This system contains the inversion formulas \eqref{eq:inversion01} 
of polylogarithms as the special case of $w=\xi_0^{k-1}\xi_1$. The duality relations of multiple zeta values
\begin{equation}\label{eq:duality_mzv}
      \zeta(\reg^{10}(w)) = \zeta(\reg^{10}(\tau(w)))
\end{equation}
is included in these relations.

In this article, we consider an inverse problem for the generalized inversion formulas \eqref{eq:gif_ext}, 
and show that this problem is nothing but a recursive Riemann-Hilbert problem (or a Plemelj-Birkhoff decomposition) 
of additive type \cite{Bi}, \cite{M}, \cite{P}. For any word $w$ of $\xi_0, \, \xi_1$, and a certain asymptotic condition, 
this problem has a unique solution (Theorem \ref{thm:MPL_RH}). \\

Furthermore this result can be interpreted as an inverse problem for the connection problem of the KZ equation 
\eqref{eq:connection01}. 
Thus one can establish the existence and uniqueness of a solution $F^{(0)}(z)$ and $F^{(1)}(z)$ to the equation
\begin{equation}\label{eq:mRH}
                \left(F^{(1)}(z)\right)^{-1}F^{(0)}(z)=\varPhi_{\rm KZ}
\end{equation}
under certain assumptions (Theorem \ref{thm:KZ_RH}). 
In other words, the solutions $\cL^{(0)}(z)$ and $\cL^{(1)}(z)$ are completely characterized by 
the Riemann-Hilbert problem of multiplicative type \eqref{eq:mRH}.\\

This article is organized as follows: In Section 2, we prepare some terminologies about free Lie algebras 
and shuffle algebras, and survey the connection problem of the KZ equation of one variable due to \cite{OkU} and 
\cite{OiU1}. In Section 3, we consider in details the generalized inversion formulas. 
We prove the generalized inversion formulas independently of the connection relation \eqref{eq:connection01}, 
and show that the duality relation \eqref{eq:duality_mzv} is obtained as the consistency condition. 
In Section 4, we formulate and solve the recursive Riemann-Hilbert problem of additive type corresponding to 
the inverse problem for the generalized inversion formulas.
Finally, in Section 5, we consider the inverse problem for the connection problem of the KZ equation 
and show that it has a unique solution. 

\paragraph{\bf Acknowledgment}

The second author is supported by JSPS KAKENHI Grant Number 25400054.

\section{The connection problem of the KZ equation of one variable}\label{sec:prep}

Let $\fX=\C\{X_0,X_1\}$ be the free Lie algebra generated by $X_0$ and $X_1$ and $\cU=\cU(\fX)$ be 
the universal enveloping algebra of $\fX$. $\cU$ is the non-commutative polynomial algebra 
$\C\langle X_0,X_1 \rangle$ generated by $X_0, X_1$ with the unit $\bunit$.

The algebra $\cU$ has a co-commutative Hopf algebra structure by the following coproduct $\Delta$, 
the counit $\varepsilon$ as algebra morphisms and the antipode $\rho$ as an anti-algebra morphism:
\begin{align*}
\Delta(X_i)&=\bunit\otimes X_i+X_i\otimes \bunit,\\
\varepsilon(X_i)&=0,\\
\rho(X_i)&=-X_i.
\end{align*}

The Hopf algebra $\cU$ has a grading defined by the length of words; 
\begin{equation}
        \cU=\bigoplus_{s=0}^\infty \cU_s.
\end{equation}
We also denote by $\tcU=\tcU(\fX)$ the completion of $\cU$ with respect to this grading. 
$\tcU$ is the non-commutative formal power series algebra $\C\langle\langle X_0,X_1 \rangle\rangle$.

Let $S=S(\xi_0,\xi_1)$ be a shuffle algebra generated by 1-forms
\begin{equation}
\xi_0=\frac{dz}{z},\qquad \xi_1=\frac{dz}{1-z}.
\end{equation}
This is the non-commutative polynomial algebra $\C\langle \xi_0,\xi_1 \rangle$ generated by $\xi_0,\xi_1$ 
with the shuffle product $\sh$. The shuffle product is defined recursively as
\begin{align*}
w \sh \bnull&=\bnull\sh w=w,\\
(\xi_{i_1} w_1) \sh (\xi_{i_2} w_2)&=\xi_{i_1} (w_1 \sh (\xi_{i_2} w_2)) + \xi_{i_2} ((\xi_{i_1} w_1) \sh w_2),
\end{align*}
where $\bnull$ is the unit of $\C\langle \xi_0,\xi_1 \rangle$ (that is, $\bnull$ stands for the empty word) 
and $w,w_1,w_2$ are words of $\C\langle \xi_0,\xi_1 \rangle$.

By virtue of Reutenauer \cite{R}, $S$ is an associative commutative algebra and has a Hopf algebra structure by the coproduct
\begin{equation*}
\Delta^{*}(\xi_{i_1}\cdots \xi_{i_r})=\sum_{k=0}^r \xi_{i_1}\cdots \xi_{i_k}\otimes \xi_{i_{k+1}}\cdots \xi_{i_{r}}
\end{equation*}
(we regard $\xi_{i_1}\cdots \xi_{i_0}$ (at $k=0$) and $\xi_{i_{r+1}}\cdots \xi_{i_r}$ (at $k=r$) as $\bnull$), 
the counit $\varepsilon^{*}(\xi_i)=0$ and the antipode 
$\rho^{*}(\xi_{i_1}\cdots \xi_{i_r})=(-1)^r \xi_{i_r}\cdots \xi_{i_1}$.

The shuffle algebra
\begin{equation}
S=\bigoplus_{s=0}^{\infty}S_s
\end{equation}
is also a graded Hopf algebra with the grading defined by the length of words. The dual of this Hopf algebra is 
the algebra $\widetilde{\cU}$ defined above with respect to the pairing
\begin{equation*}
\langle \xi_{i_1}\cdots \xi_{i_r},X_{j_1}\cdots X_{j_s}\rangle=
\begin{cases}
1&(r=s, i_k=j_k \text{ for }1\le k \le r),\\
0&\text{(otherwise)}.
\end{cases}
\end{equation*}

In what follows, we denote conveniently the sum over all words $w$ in $S$ by $\ds \sum_{w \in S}$ (or similar notations), 
and the dual element of $w$ by the capital letter $W$ (that is, for $w=\xi_{i_1}\cdots\xi_{i_r} \in S$, 
the capitalization $W$ stands for $X_{i_1}\cdots X_{i_r} \in \cU$).\\

The following lemmas are basic and will be used in Section \ref{sec:KZ_RH}.

\begin{lem}\label{lem:grouplike_shufflehom}
A $\tcU$-valued function
\begin{equation*}
F(z)=\sum_{w \in S}f(w;z)W,
\end{equation*}
which is holomorphic at $z=0$ and $F(0)=\bunit$, is grouplike if and only if $f(w;z)$ is a shuffle homomorphism.
\end{lem}

\begin{lem}\label{lem:grouplike_reciprocal}
If a $\tcU$-valued function
\begin{equation*}
F(z)=\sum_{w \in S}f(w;z)W
\end{equation*}
is grouplike, holomorphic at $z=0$ and $F(0)=\bunit$, the reciprocal of $F(z)$ is written as
\begin{equation}
\left(F(z)\right)^{-1}=\sum_{w \in S}f(w;z)\rho(W)=\sum_{w \in S}f(\rho^*(w);z)W.
\end{equation}
\end{lem}

\vspace{1\baselineskip}

We denote by $S^0$ and $S^{10}$ the subalgebras of $S$ defined as
\begin{align*}
S^0&=\C\bnull+S \xi_1,\\
S^{10}&=\C\bnull+\xi_0 S \xi_1.
\end{align*}

$S$ has polynomial ring structures as follows:
\begin{prop}[\cite{R}]\label{prop:reutanuer}
$S$ is a polynomial algebra of $\xi_0$ whose coefficients are in $S^0$, and is a polynomial algebra 
of $\xi_0, \xi_1$ whose coefficients are in $S^{10}$;
\begin{align}\label{poly_rep}
S=S^0[\xi_0]=S^{10}[\xi_0,\xi_1].
\end{align}
That is, any word $w$ in $S$ can be written as
\begin{equation}\label{sh_decomposition}
w=\sum_{j}w_j \sh \xi_0^j = \sum_{i,j}\xi_1^i \sh w_{ij} \sh \xi_0^j
\end{equation}
uniquely, where $w_j \in S^1$ and $w_{ij} \in S^{10}$.
\end{prop}

We define the regularizing maps $\reg^0$ and $\reg^{10}$ as picking up the constant terms of a word 
with respect to the decomposition \eqref{sh_decomposition}:
\begin{align*}
& \reg^0: S=S^0[\xi_0] \to S^0, \quad u=\sum w_j\sh \xi_0^j \mapsto w_0 \ (w_j \in S^0), \\
& \reg^{10}: S=S^{10}[\xi_0,\xi_1] \to S^{10}, \quad 
u=\sum \xi_1^i\sh w_{ij}\sh \xi_0^j \mapsto w_{00} \ (w_{ij} \in S^{10}).
\end{align*}

The maps $\reg^0$ and $\reg^{10}$ are shuffle homomorphisms and are calculated by the following lemma. 
This lemma was implicitly shown in \cite{IKZ}. 

\begin{lem}[\cite{IKZ}]\label{lem:reg_calc}
\begin{enumerate}
\item For a word $u \in S^0$, we have
\begin{align}
u\xi_0^n&=\sum_{j=0}^n \reg^0(u \xi_0^{n-j})\sh \xi_0^j, \label{lem:reg_calc1}\\
\reg^0(u \xi_0^n)&=\sum_{j=0}^n (-1)^j(u\xi_0^{n-j})\sh \xi_0^j. \label{lem:reg_calc2}
\end{align}

\item For a word $u \in S^{10}$, we have
\begin{align}
\xi_1^m u\xi_0^n&=\sum_{i=0}^m\sum_{j=0}^n \xi_1^i\sh\reg^{10}(\xi_1^{m-i}u \xi_0^{n-j})\sh \xi_0^j, \label{lem:reg_calc3}\\
\reg^{10}(\xi_1^m u \xi_0^n)&=\sum_{i=0}^m \sum_{j=0}^n (-1)^{i+j}\xi_1^i\sh(\xi_1^{m-i}u\xi_0^{n-j})\sh \xi_0^j. \label{lem:reg_calc4}
\end{align}
\end{enumerate}
\end{lem}

\vspace{1\baselineskip}

Let $\D_0$ and $\D_1$ be domains on $\C$ defined by
\begin{align}
\D_0&=\C\setminus \{z=x\;|\; x\ge 1\},\\
\D_1&=\C\setminus \{z=x\;|\; x\le 0\}.
\end{align}

For a word $w=\xi_{i_1}\cdots \xi_{i_r} \in S^0$ (that is, $i_r=1$), we define an iterated integral
\begin{equation}\label{eq:iterated_integral}
\int_0^z w=\begin{cases}
1&(w=\bnull),\\
\ds \int_0^z \xi_{i_1}\left(\ds \int_0^z \xi_{i_2}\cdots \xi_{i_r}\right) &(\text{otherwise})
\end{cases}
\end{equation}
recursively and extend it to $S^0$ as a linear map.\\

For $z \in \D_0$ and $w=\xi_{i_1}\cdots\xi_{i_r} \in S^0$, we define the multiple polylogarithms 
of one variable $\Li(w;z)$ as 
\begin{equation}\label{eq:MPLdef}
\Li(w;z)=\int_0^z w=\int_0^z \xi_{i_1}\cdots\xi_{i_r}.
\end{equation}
These are holomorphic functions on $\D_0$ and have Taylor expansions
\begin{align}\label{MPL_Taylor}
\Li(\xi_0^{k_1-1}\xi_1\cdots\xi_0^{k_r-1}\xi_1;z)
=\sum_{n_1>n_2\cdots>n_r>0}\frac{z^{n_1}}{n_1^{k_1}\cdots n_r^{k_r}} \quad (|z|<1),
\end{align}
which coincides with $\Li_{k_1,\dots,k_r}(z)$.
In what follows, we consider the multiple polylogarithms only of one variable, so we omit ``of one variable.''

We extend $\Li(w;z)$ to $S$ as follows: For a word $w = \sum_j w_j\sh \xi_0^j$ ($w_j \in S^0$), 
we set an extended multiple polylogarithm by
\begin{equation}\label{eq:extendedMPL}
     \Li(w;z)=\sum_j\Li(w_j;z)\frac{\log^j z}{j!}
\end{equation}
Here $\log z$ is defined as the principal value on $\D_1$. 
These extended multiple polylogarithms are holomorphic on $\D_0\cap \D_1$, and the map 
\begin{equation}
\Li(\bullet;z) \,:\, S \longrightarrow \C, \quad w \longmapsto \Li(w;z) \quad (z \in \D_0\cap \D_1)
\end{equation}
is a shuffle homomorphism.

\begin{lem}[\cite{HPH}\cite{Ok}]\label{lem:MPL_diff}
The extended multiple polylogarithms satisfy the following recursive differential relations:
\begin{equation}\label{eq:MPL_diff}
\frac{d}{dz}\Li(\xi_0u;z)=\frac{\Li(u;z)}{z},\qquad \frac{d}{dz}\Li(\xi_1u;z)=\frac{\Li(u;z)}{1-z},
\end{equation}
where $u$ is a word of $S$.
\end{lem}

We observe that the extended multiple polylogarithms can be continued  onto 
$\bP^1\setminus\{0,1,\infty\}$ as many-valued analytic functions
along the integral paths of \eqref{eq:MPLdef}.\\

If $w=\xi_0^{k_1-1}\xi_1\cdots\xi_0^{k_r-1}\xi_1 \in S^{10}$, the limit of $\Li(w;z)$ as $z$ tends to 1 in $\D_0\cap \D_1$ 
converges and defines multiple zeta values:
\begin{equation}\label{eq:MZV}
\zeta(w)=\lim_{\substack{z \to 1\\z \in \D_0\cap \D_1}}\Li(w;z)=\sum_{n_1>n_2\cdots>n_r>0}\frac{1}{n_1^{k_1}\cdots n_r^{k_r}}.
\end{equation}
The multiple zeta values $\zeta(w)$ are denoted by $\zeta(k_1,\ldots,k_r)$ as usual.\\

Under these notations, the specific solution to the equation \eqref{eq:1KZ} 
can be written as follows:
\begin{prop}[\cite{OiU1}, \cite{OkU}]\label{prop:KZE1}
The KZ equation of one variable \eqref{eq:1KZ} has the solution $\cL^{(0)}(z)$ which satisfies the asymptotic condition
\begin{equation*}
\cL^{(0)}(z)=\hcL^{(0)}(z)z^{X_0},
\end{equation*}
where  $\hcL^{(0)}(z)$ is holomorphic at $z=0$ and $\hcL^{(0)}(0)=\bunit$.

The solution $\cL^{(0)}(z)$ is uniquely characterized by this condition and is a grouplike element of $\tcU$.

Furthermore the solution $\cL^{(0)}(z)$ is expressed as
\begin{align}
\cL^{(0)}(z)&=\sum_{w \in S}\Li(w;z)W \label{eq:fundamental0}\\
&=\Big(\sum_{w \in S}\Li(\reg^{0}(w);z)W \Big)z^{X_0} \label{eq:fundamental0_0}\\
&=(1-z)^{-X_1}\Big(\sum_{w \in S}\Li(\reg^{10}(w);z)W \Big)z^{X_0}. \label{eq:fundamental0_00}
\end{align}
\end{prop}
We call the solution $\cL^{(0)}(z)$ the fundamental solution of \eqref{eq:1KZ} normalized at $z=0$. We also refer to the solution 
\begin{equation}
\cL^{(1)}(z)=\hcL^{(1)}(z)(1-z)^{-X_1}, 
\end{equation}
where $\hcL^{(1)}(z)$ is holomorphic at $z=1$, $\hcL^{(1)}(1)=\bunit$, 
as the fundamental solution normalized at $z=1$.\\

With respect to the transformation $t: z \mapsto 1-z$, we introduce the automorphism $t^*$ on $S$ by
\begin{equation}
t^*(\xi_0)=-\xi_1,\qquad t^*(\xi_1)=-\xi_0,
\end{equation}
which is the pull back induced from $t$, and also introduce the automorphism $t_*$ on $\cU$ by
\begin{equation}
t_*(X_0)=-X_1,\qquad t_*(X_1)=-X_0,
\end{equation} 
which is the dual map of $t^*$. Let $\tau : S \to S$ be an anti-automorphism defined by $\tau=t^*\circ \rho^*$, that is, 
\begin{equation}
\tau(\xi_0)=\xi_1,\qquad \tau(\xi_1)=\xi_0.
\end{equation}
Furthermore put $T=t_*\circ \rho$ which is an anti-automorphism on $\cU$ satisfying 
\begin{equation}
           T(X_0)=X_1, \qquad T(X_1)=X_0. 
\end{equation}

Using the transformation $t$ and the automorphism $t_*$ , 
the fundamental solution $\cL^{(1)}$ of the KZ equation \eqref{eq:1KZ} 
normalized at $z=1$ is written as
\begin{align}
\cL^{(1)}(z)
&=\sum_{w \in S}\Li(w;1-z)t_*(W) \label{eq:fundamental1}\\
&=\Big(\sum_{w \in S}\Li(\reg^{0}(w);1-z)t_*(W)\Big)(1-z)^{-X_1} \label{eq:fundamental1_0}\\
&=z^{X_0}\Big(\sum_{w \in S}\Li(\reg^{10}(w);1-z)t_*(W)\Big)(1-z)^{-X_1}. \label{eq:fundamental1_00}
\end{align}
The connection relation between $\cL^{(0)}$ and $\cL^{(1)}$ is described as follows:

\begin{prop}[\cite{D}\cite{OkU}]\label{prop:connection1}
\begin{enumerate}
\item The connection matrix between $\cL^{(0)}(z)$ and $\cL^{(1)}(z)$ is given by the Drinfel'd associator
\begin{equation}\label{eq:DA}
\varPhi_{\rm KZ}=\varPhi_{\rm KZ}(X_0,X_1)=\sum_{w\in S}\zeta(\reg^{10}(w))W.
\end{equation}
That is, the connection formula reads 
\begin{equation}\label{eq:connection01'}
\cL^{(0)}(z)=\cL^{(1)}(z)\varPhi_{\rm KZ}.
\end{equation}
\item The connection formula \eqref{eq:connection01'} is equivalent to the system of relations
\begin{align}\label{eq:gif_}
\sum_{uv=w}\Li(\tau(u);1-z)\Li(v;z)=\zeta(\reg^{10}(w))
\end{align}
for all words $w \in S$.
\end{enumerate}
\end{prop}

We call the relations \eqref{eq:gif_} the generalized inversion formulas for extended multiple polylogarithms.\\

This proposition follows from the representation \eqref{eq:fundamental0_00}, \eqref{eq:fundamental1_00}, 
Lemma \ref{lem:grouplike_reciprocal}, and the definition \eqref{eq:MZV} of the multiple zeta values. 

\section{The generalized inversion formulas for the extended multiple polylogarithms}

According to Proposition \ref{prop:connection1}, the generalized inversion formulas \eqref{eq:gif_} 
is equivalent to the connection problem of the KZ equation of one variable. 
However we can show these formulas independently of the connection problem of the KZ equation.\\


\begin{prop}\label{prop:gif}
For any word $w \in S$, the generalized inversion formula
\begin{align}\label{eq:gif}
\sum_{uv=w}\Li(\tau(u);1-z)\Li(v;z)=\zeta(\reg^{10}(w))
\end{align}
holds.
\end{prop}

To prove this, it is enough to show the following lemma. This lemma also plays a key role to prove 
Theorem \ref{thm:MPL_RH} in Section \ref{sec:MPL_RH}.

\begin{lem}\label{lem:gen_inv_limit}
\begin{enumerate}
\item\label{lem:gen_inv_limit_1} We have 
\begin{equation}
\frac{d}{dz}\left(\sum_{uv=w}\Li(\tau(u);1-z)\Li(v;z)\right)=0. 
\end{equation}
\item\label{lem:gen_inv_limit_2} For any word $w$ in $S$, we have
\begin{align}
\lim_{\substack{z \to 1,\\z \in \D_0\cap \D_1}}\sum_{uv=w}\Li(\tau(u);1-z)\Li(v;z)&=\zeta(\reg^{10}(w)), 
                                                                         \label{eq:lim_gif_1}\\
\lim_{\substack{z \to 0,\\z \in \D_0\cap \D_1}}\sum_{uv=w}\Li(\tau(u);1-z)\Li(v;z)&=\zeta(\reg^{10}(\tau(w)). 
                                                                         \label{eq:lim_gif_0}
\end{align}
\end{enumerate}
\end{lem}

\begin{proof}
\eqref{lem:gen_inv_limit_1}\ We represent the differential recursive relations $\eqref{eq:MPL_diff}$ 
in terms of the exterior derivative with respect to the variable $z$; 
\begin{align*}
d\Li(\xi_iw;z) = \xi_i \Li(w;z) \qquad (i=0,1).
\end{align*}
From this, it follows that
\begin{align*}
d\Li(\tau(\xi_i)w;1-z) = -\xi_i \Li(w;z) \qquad (i=0,1).
\end{align*}
Hence, for a word $w=\xi_{i_1} \cdots \xi_{i_r}$, we have 
\begin{align*}
&d\left(\sum_{uv=w}\Li(\tau(u);1-z)\Li(v;z)\right) \\
=\ &d\left(\sum_{k=0}^r\Li(\tau(\xi_{i_1}\cdots\xi_{i_k});1-z)\Li(\xi_{i_{k+1}}\cdots\xi_{i_r};z)\right) \\
=\ &\xi_{i_1}\Li(\xi_{i_2}\cdots\xi_{i_r};z) \\
&+\sum_{k=1}^{r-1}\Big(-\xi_{i_k}\Li(\tau(\xi_{i_1}\cdots\xi_{i_{k-1}});1-z) \Li(\xi_{i_{k+1}}\cdots\xi_{i_r};z)\\[-1ex]
&\hspace{4cm}+\xi_{i_{k+1}}\Li(\tau(\xi_{i_1}\cdots\xi_{i_k});1-z)\Li(\xi_{i_{k+2}}\cdots\xi_{i_r};z) \Big)\\
&-\xi_{i_r}\Li(\tau(\xi_{i_1}\cdots\xi_{i_{r-1}});1-z) \\
=\ &\xi_{i_1}\Li(\xi_{i_2}\cdots\xi_{i_r};z) -\xi_{i_1}\Li(\xi_{i_2}\cdots\xi_{i_r};z)
+ \xi_{i_r}\Li(\tau(\xi_{i_1}\cdots\xi_{i_{r-1}});1-z) \\
&-\xi_{i_r}\Li(\tau(\xi_{i_1}\cdots\xi_{i_{r-1}});1-z) \\
=\ &0. 
\end{align*}

\noindent
\eqref{lem:gen_inv_limit_2}\ For a word $w=\xi_0^r$ or $w=\xi_1^r$, the both sides of \eqref{eq:lim_gif_1} 
and \eqref{eq:lim_gif_0} are trivially zero. For a word $w=\xi_1^k w' \xi_0^l,\; w' \in S^{10}$, 
we will prove \eqref{eq:lim_gif_1}. One can similarly prove \eqref{eq:lim_gif_0}.

For $w=\xi_1^k w' \xi_0^l,\; w' \in S^{10}$, we have
\begin{align}
&\sum_{uv=w}\Li(\tau(u);1-z)\Li(v;z) \notag \\
=&\sum_{i=0}^k\Li(\tau(\xi_1^i);1-z)\Li(\xi_1^{k-i}w'\xi_0^l;z)
+\sum_{\substack{uv=w'\xi_0^l\\u\neq \bnull}}\Li(\tau(\xi_1^k u);1-z)\Li(v;z). \label{proof:gen_inv_limit:1}
\end{align}
For the second term of the right hand side of \eqref{proof:gen_inv_limit:1}, by putting $u=\xi_0u'$, we obtain
\begin{align*}
\Li(\tau(\xi_1^k u);1-z)\Li(v;z)&=\Li(\tau(\xi_1^k \xi_0 u');1-z)\Li(v;z)  \\
&=\Li(\tau(u')\xi_1\xi_0^k);1-z)\Li(v;z)\\
&=\sum_{s=0}^k \Li(\reg^0(\tau(u')\xi_1\xi_0^{k-s}));1-z)\frac{\log^s (1-z)}{s!}\Li(v;z). 
\end{align*}
Since $\reg^0(\tau(u')\xi_1\xi_0^{k-s})\neq \bnull$ so that 
\begin{equation*}
\Li(\reg^0(\tau(u')\xi_1\xi_0^{k-s}));1-z)=O(1-z) \quad (z \to 1),
\end{equation*}
and $\Li(v;z)$ diverges at most of logarithmic order as $z \to 1$, we have 
\begin{equation}
\Li(\tau(\xi_1^k u);1-z)\Li(v;z) \to 0 \quad (z \to 1). 
\end{equation}

Next, we consider the first term of the right hand side of \eqref{proof:gen_inv_limit:1}. 
By using Lemma \ref{lem:reg_calc}, we have
\begin{align}
&\sum_{i=0}^k\Li(\tau(\xi_1^i);1-z)\Li(\xi_1^{k-i}w'\xi_0^l;z) \notag \\
=&\sum_{i=0}^k\Li(\xi_0^i;1-z)\Li(\xi_1^{k-i}w'\xi_0^l;z) \notag \\
=&\sum_{i=0}^k\sum_{p=0}^{k-i}\sum_{q=0}^l\Li(\xi_0^i;1-z)\Li(\xi_1^p;z)\Li(\reg^{10}(\xi_1^{k-i-p}w'\xi_0^{l-q});z)\Li(\xi_0^q;z) \notag \\
=&\sum_{i=0}^k\sum_{p=0}^{k-i}\sum_{q=0}^l\frac{\log^i(1-z)}{i!}\frac{(-\log(1-z))^p}{p!}\Li(\reg^{10}(\xi_1^{k-i-p}w'\xi_0^{l-q});z)\frac{\log^q z}{q!} \notag \\
=&\sum_{r=0}^k\sum_{q=0}^l \left(\sum_{i+p=r}\frac{1}{i!}\frac{(-1)^p}{p!}\right)\log^r(1-z)\Li(\reg^{10}(\xi_1^{k-r}w'\xi_0^{l-q});z)\frac{\log^q z}{q!}. \label{proof:gen_inv_limit:2}
\end{align}
From the identity
\begin{equation}
\sum_{i+p=r}\frac{1}{i!}\frac{(-1)^p}{p!}=\begin{cases}
1&(r=0), \\
0&(r\neq 1), 
\end{cases}
\end{equation}
it follows that 
\begin{align}
\eqref{proof:gen_inv_limit:2}&=\sum_{q=0}^l \Li(\reg^{10}(\xi_1^{k}w'\xi_0^{l-q});z)\frac{\log^q z}{q!} \notag \\
&\to \Li(\reg^{10}(\xi_1^kw'\xi_0^l);1)=\zeta(\reg^{10}(\xi_1^kw'\xi_0^l)) \qquad (z \to 1).
\end{align}
\end{proof}

We should observe that the generalized inversion formulas are overdetermined and contains some relations of multiple zeta values. For instance, replacing $w$ to $\tau(w)$ in \eqref{eq:gif}, we have 
\begin{align*}
\sum_{uv=\tau(w)}\Li(\tau(u);1-z) \Li(v;z)=\zeta(\reg^{10}(\tau(w))). 
\end{align*}
Furthermore, in this formula, replace $z$ to $1-z$ and put $u=\tau(v'), v=\tau(u')$. Then we obtain 
\begin{align}\label{eq:gif_tau}
\sum_{u'v'=w} \Li(v';z) \Li(\tau(u');1-z) = \zeta(\reg^{10}(\tau(w))). 
\end{align}
The left hand side of \eqref{eq:gif_tau} coincides with the left hand side of \eqref{eq:gif}, 
so that we have the duality relations for multiple zeta values
\begin{align}\label{eq:dual_mzv}
\zeta(\reg^{10}(w)) = \zeta(\reg^{10}(\tau(w))) 
\end{align}
for any word $w$ in $S$.

\section{The Riemann-Hilbert problem of additive type for multiple polylogarithms}\label{sec:MPL_RH}

In this section, we solve the recursive Riemann-Hilbert problem of additive type corresponding to the inverse problem of the generalized inversion formulas \eqref{eq:gif}.\\

Note that the generalized inversion formulas \eqref{eq:gif} for a word $w \in S^{10}$ read
\begin{align}
&\Li(\tau(w);1-z)+\Li(w;z)=\zeta(w)-\sum_{\substack{uv=w\\u,v \neq \bnull}}\Li(\tau(u);1-z)\Li(v;z).
\end{align}
The equation says that the right hand side, which is holomorphic on $\D_{0} \cap \D_{1}$, decomposes to 
the sum of $\Li(\tau(w);1-z)$ and $\Li(w;z)$, which are holomorphic on $\D_{0}$ and $\D_{1}$ respectively. Moreover the length of words appeared in the right hand side are less than the length of the word $w$. Hence this decomposition is considered as a recursive Riemann-Hilbert problem of additive type.

\begin{thm}\label{thm:MPL_RH}
There exist uniquely 
$\sh$-homomorphisms $f^{(0)}(\bullet;z),\; f^{(1)}(\bullet;z): S \to \C$, which satisfy
\begin{equation}
f^{(0)}(\xi_0;z)=\log z, \quad  f^{(1)}(\xi_0;z)=\log (1-z),
\end{equation}
and the following three conditions:
\begin{enumerate}
\item For any word $w \in S$, $f^{(0)}(w;z)$ and $f^{(1)}(w;z)$ enjoy the functional equations
\begin{equation}\label{aRH}
\sum_{uv=w} f^{(1)}(\tau(u);z) f^{(0)}(v;z) = \zeta(\reg^{10}(w)) \quad  (z \in \D_{0} \cap \D_{1}).
\end{equation}
\item For any word $w \in S^{10}$, $f^{(0)}(w;z)$ and $f^{(1)}(w;z)$ are holomorphic on $\D_0$ and $\D_1$ respectively and satisfy the asymptotic conditions
\begin{equation}\label{Asym}
\frac{d}{dz}f^{(i)}\big(w;z\big) \to 0 \quad (z \to \infty, \ z \in \D_{i}).
\end{equation}
\item For any word $w \in S^{10}$, $f^{(0)}(w;z)$ satisfies the normalizing conditions
\begin{equation}\label{Norm}
f^{(0)}\big(w;0  \big)=0.
\end{equation}
\end{enumerate}

The solutions $f^{(i)}(\bullet;z)$ are expressed in terms of extended multiple polylogarithms as follows; 
\begin{align}
f^{(0)}(w;z)&=\Li(w;z),\\
f^{(1)}(w;z)&=\Li(w;1-z).
\end{align}
\end{thm}

\begin{proof}
By Proposition \ref{prop:gif}, the functions $f^{(0)}(w;z)=\Li(w;z)$ and $f^{(1)}(w;z)=\Li(w;1-z)$ satisfy all of the previous conditions.\\

We show that $f^{(i)}(w;z)$ are uniquely determined by using induction on the length of a word $w$.

First, in the case of $w=\xi_0$, the equation \eqref{aRH} reads
\begin{equation}
f^{(1)}(\xi_1;z)+f^{(0)}(\xi_0;z) = 0.
\end{equation}
Therefore we obtain
\begin{equation}\label{proof:MPL_RH:f1_xi1}
f^{(1)}(\xi_1;z)=-f^{(0)}(\xi_0;z)=-\log z=\Li(\xi_1;1-z).
\end{equation}
In a similar fashion, we have
\begin{equation}\label{proof:MPL_RH:f0_xi1}
f^{(0)}(\xi_1;z)=-f^{(1)}(\xi_0;z)=-\log(1-z)=\Li(\xi_1;z)
\end{equation}
in the case of $w=\xi_1$.\\

Next, we assume that $f^{(0)}(w';z)=\Li(w';z)$ and $f^{(1)}(w';z)=\Li(w';1-z)$ for words $w'$ whose length is less than $r$.

Now, if the result $f^{(0)}=\Li(w;z)$ and $f^{(1)}=\Li(w;1-z)$ hold for all words $w \in S^{10}$ of length $r$, we obtain
\begin{align*}
f^{(0)}(w;z)&=f(\sum_{i,j}\xi_1^i\sh w_{ij} \sh \xi_0^j; z)\\
&=\sum_{i,j}\frac{f(\xi_1;z)^i}{i!}f(w_{ij};z)\frac{f(\xi_0; z)^j}{j!}\\
&=\sum_{i,j}\frac{\Li(\xi_1;z)^i}{i!}\Li(w_{ij};z)\frac{\Li(\xi_0; z)^j}{j!}\\
&=\Li(\sum_{i,j}\xi_1^i\sh w_{ij} \sh \xi_0^j; z)=\Li(w;z)
\end{align*}
for any word $w \in S$ of length $r$, since $f^{(0)}(w;z)$ and $\Li(w;z)$ are both shuffle homomorphisms and the word $w$ has the unique decomposition \eqref{sh_decomposition}. In the similar way, we have also $f^{(1)}(w;z)=\Li(w;1-z)$ for any word $w \in S$.

Therefore it suffice to show that $f^{(0)}(w;z)=\Li(w;z)$ and $f^{(0)}(w;z)=\Li(w;1-z)$ for any word $w \in S^{10}$ of length $r$.\\

Let $w=\xi_{i_1}\cdots \xi_{i_r}$ be a word of $S^{10}$ of length $r$. Under the assumption of induction, the equation \eqref{aRH} becomes
\begin{align}\label{aRH2}
& f^{(1)}(\tau(w);z)+\sum_{k=1}^{r-1} \Li(\tau(\xi_{i_1}\cdots \xi_{i_k});1-z)\Li(\xi_{i_{k+1}} \cdots \xi_{i_r};z)+f^{(0)}(w;z) \notag \\
=&\zeta(\reg^{10}(w))=\zeta(w).
\end{align}

Now we show $f^{(1)}(w;z)=\Li(w;1-z)$ and $f^{(0)}(w;z)=\Li(w;z)$ by using \eqref{Asym}, \eqref{Norm} 
and \eqref{aRH2}. According to Lemma \ref{lem:gen_inv_limit} \eqref{lem:gen_inv_limit_1} we have
\begin{align*}
&d\left(\sum_{k=1}^{r-1} \Li(\tau(\xi_{i_1}\cdots \xi_{i_k});1-z)\Li(\xi_{i_{k+1}}\cdots \xi_{i_r};z)\right)\\
=&-d\Li(\xi_{i_1}\cdots \xi_{i_r};z)-d\Li(\tau(\xi_{i_1}\cdots \xi_{i_{r}});1-z).
\end{align*}
Thus the differentiation of \eqref{aRH2} leads to the equation
\begin{equation}\label{aRH2_diff}
df^{(0)}(w;z)-d\Li(w;z)=-df^{(1)}(\tau(w);z)+d\Li(\tau(w);1-z).
\end{equation}
Here we notice that both $w$ and $\tau(w)$ are words in $S^{10}$.

Since the left hand side (resp. right hand side) of \eqref{aRH2_diff} is holomorphic on $\D_0$ (resp. $\D_1$), 
the both side of \eqref{aRH2_diff} are entire functions. By \eqref{Asym}, due to Liouville's theorem, we obtain
\begin{gather*}
df^{(0)}(w;z)-d\Li(w;z)=0,\\
df^{(1)}(\tau(w);z)-d\Li(\tau(w);1-z)=0.
\end{gather*}
Thus the functions $f^{(0)}(w;z)$ and $f^{(1)}(w;z)$ are determined as
\begin{align}
f^{(0)}(w;z)&=\Li(w;z)+c^{(0)}(w),\label{aRH_constant_0} \\
f^{(1)}(\tau(w);z)&=\Li(\tau(w);1-z)+c^{(1)}(w). \label{aRH_constant_1}
\end{align}
where $c^{(0)}(w)$ and $c^{(1)}(w)$ are integration constants. 

Finally, we determine the integral constants $c^{(0)}(w)$ and $c^{(1)}(w)$. By \eqref{Norm}, $c^{(0)}(w)=0$ is clear. Substituting \eqref{aRH_constant_0} and \eqref{aRH_constant_1} to \eqref{aRH}, we have
\begin{equation*}
\sum_{uv=w}\Li(\tau(u);1-z)\Li(v;z)+c^{(1)}(w)=\zeta(w).
\end{equation*}
In this relation, letting $z \to 1 (z \in \D_0 \cap \D_1)$, we obtain, from Lemma \ref{lem:gen_inv_limit} \eqref{eq:lim_gif_1},
\begin{equation}
\zeta(w)+c^{(1)}(w)=\zeta(w).
\end{equation}
Thus we have
\begin{equation}
c^{(1)}(w)=0
\end{equation}
and have completed the proof of this theorem.

\end{proof}

We note that relations among multiple zeta values by derived from the generalized inversion formula (e.g. the duality relation \eqref{eq:dual_mzv}) can be interpreted as the consistency condition for existing the solution of this Riemann-Hilbert problem \eqref{aRH}.

\section{The Riemann-Hilbert problem corresponding to the KZ equation of one variable}\label{sec:KZ_RH}

In this section, we consider the Riemann-Hilbert problem as the inverse problem of the connection problem 
of the KZ equation of one variable \eqref{eq:connection01'} and discuss the relationship 
between this Riemann-Hilbert problem and Theorem \ref{thm:MPL_RH}.\\

The connection formula \eqref{eq:connection01'} can be written as
\begin{equation}\label{eq:connection01''}
\left(\cL^{(1)}(z)\right)^{-1}\cL^{(0)}(z)=\varPhi_{\rm KZ}.
\end{equation}
We show that the fundamental solutions $\cL^{(0)}(z)$ and $\cL^{(1)}(z)$ of the KZ equation \eqref{eq:1KZ} are determined by this equation as a Riemann-Hilbert problem of multiplicative type.

The Riemann-Hilbert problem corresponding to the KZ equation is to find $\tcU$-valued functions $F^{(0)}(z)$ and $F^{(1)}(z)$ which are grouplike, 
satisfy the relation
\begin{equation}
\left(F^{(1)}(z)\right)^{-1}F^{(0)}(z)=\varPhi_{\rm KZ}.
\end{equation}
in $\D_0\cap \D_1$ and some conditions. Applying Theorem \ref{thm:MPL_RH}, we can find a solution to this Riemann-Hilbert problem. \\

Before stating the theorem, we mention the following proposition which will be used in the comment after the proof of the theorem.

\begin{prop}[\cite{OkU}]\label{prop:dual_dr}
The duality relations \eqref{eq:dual_mzv} of the multiple zeta values are equivalent to the duality 
of the Drinfel'd associator
\begin{equation}\label{eq:dual_dr}
\varPhi_{\rm KZ}(X_0,X_1) = \Big(\varPhi_{\rm KZ}(-X_1,-X_0)\Big)^{-1}. 
\end{equation}
\end{prop}

\begin{proof}
Since the Drinfel'd associator is a grouplike element, we have 
\begin{align*}
\Big(\varPhi_{\rm KZ}(-X_1,-X_0)\Big)^{-1} &= \Big( \sum_{w \in S} \zeta(\reg^{10}(w)) t_*(W) \Big)^{-1} \\
& = \sum_{w \in S} \zeta(\reg^{10}(w)) T(W) \\
& = \sum_{w \in S} \zeta(\reg^{10}(\tau(w))) W.
\end{align*}
Hence \eqref{eq:dual_dr} is equivalent to \eqref{eq:dual_mzv}. 
\end{proof}

The relation \eqref{eq:dual_dr} follows from the connection problem \eqref{eq:connection01''} by changing $z$ to $1-z$ and taking the reciprocal.\\

Now we can formulate and solve the Riemann-Hilbert problem, and reconstruct the fundamental solutions of 
the KZ equation \eqref{eq:1KZ} normalized at $z=0$ and $1$ from the Drinfel'd associator $\varPhi_{\rm KZ}$ 
as the following theorem.

\begin{thm}\label{thm:KZ_RH}
There exist uniquely the $\tcU$-valued functions 
$F^{(0)}(z)$ and $F^{(1)}(z)$ defined by 
\begin{align}
F^{(0)}&=(1-z)^{-X_1}\tF^{(0)}(z)z^{X_0},&\tF^{(0)}(z)&=\sum_{w \in S}\tf^{(0)}(\reg^{10}(w);z)W, \label{hF_reg0}\\
F^{(1)}&=z^{X_0}\tF^{(1)}(z)(1-z)^{-X_1},&\tF^{(1)}(z)&=\sum_{w \in S}\tf^{(1)}(\reg^{10}(w);z)W, \label{hF_reg1}
\end{align}
which are grouplike, and enjoy the following conditions:
\begin{enumerate}
\item $F^{(0)}(z)$ and $F^{(1)}(z)$ satisfy the functional equation
\begin{equation}\label{RH}
\left(F^{(1)}(z)\right)^{-1}F^{(0)}(z)=\varPhi_{\rm KZ}.
\end{equation}
\item $\tF^{(0)}(z)$ and $\tF^{(1)}(z)$ are holomorphic on $\D_0$ and $\D_1$ respectively and satisfy the asymptotic condition
\begin{equation}\label{Asym_F}
\frac{d}{dz}\tF^{(i)}(z)\to \bunit \qquad (z \to \infty, z \in \D_i).
\end{equation}
\item $\tF^{(0)}(z)$ satisfies the normalizing condition
\begin{equation}\label{Norm_F}
\tF^{(0)}(0)=\bunit.
\end{equation}
\end{enumerate}

Then the functions $F^{(0)}(z)$ and $F^{(1)}(z)$ give the fundamental solutions of the KZ equation of one variable normalized at $z=0$ and $1$ respectively.
\end{thm}

\begin{proof}
We reduce this problem to Theorem \ref{thm:MPL_RH}. Put
\begin{align}
F^{(0)}(z)&=\sum_{w \in S}f^{(0)}(w;z)W,\\
F^{(1)}(z)&=\sum_{w \in S}f^{(1)}(t^{*}(w);z)W.
\end{align}
Since $F^{(0)}(z)$ and $F^{(1)}(z)$ are grouplike, by virtue of Lemma \ref{lem:grouplike_shufflehom} and \ref{lem:grouplike_reciprocal}, the functions $f^{(i)}(\bullet;z)$ are regarded as shuffle homomorphisms from $S$ to $\C$, and the reciprocal of $F^{(1)}(z)$ is given by
\begin{equation*}
\left(F^{(1)}(z)\right)^{-1}=\sum_{w \in S}f^{(1)}(t^{*}\circ\rho^*(w);z)W=\sum_{w \in S}f^{(1)}(\tau(w);z)W.
\end{equation*}

Under these notation, the equation \eqref{RH} can be written as
\begin{equation*}
\left(\sum_{u \in S} f^{(1)}(\tau(u);z) U\right)\left(\sum_{v \in S}f^{(0)}(v;z)V \right) = \varPhi_{\rm KZ}=\sum_{w \in S} \zeta(\reg^{10}(w))W.
\end{equation*}
The coefficient of $W$ of this equation is the equation \eqref{aRH}.\\

Next, since $f^{(0)}(w;z)$ are shuffle homomorphism and the equation \eqref{lem:reg_calc3} holds, we have
\begin{align}
F^{(0)}(z)&=\sum_{w}f^{(0)}(w;z)W \notag\\
&=\sum_{i,j}\sum_{w \in S^{10}}f^{(0)}(\xi_1^iw\xi_0^j;z)X_1^iWX_0^j \notag\\
&=\sum_{i,j}\sum_{w \in S^{10}}\sum_{s=0}^i\sum_{t=0}^jf^{(0)}(\xi_1^s;z)f^{(0)}(\reg^{10}(\xi_1^{i-s}w\xi_0^{j-t});z)f^{(0)}(\xi_0^{t};z)X_1^iWX_0^j \notag \\
&=\left(\sum_i \frac{f(\xi_1;z)^i}{i!}X_1^i\right)\left(\sum_{w \in S}f^{(0)}(\reg^{10}(w);z)W\right)\left(\sum_j \frac{f(\xi_0;z)^j}{j!}X_0^j\right).
\end{align}
Comparing this formula and \eqref{hF_reg0}, we have
\begin{equation*}
f^{(0)}(\xi_0;z)=\log z
\end{equation*}
as the coefficient of $X_0$, 
\begin{equation*}
f^{(0)}(\xi_1;z)=-\log(1-z)
\end{equation*}
as the coefficient of $X_1$, and
\begin{equation*}
f^{(0)}(w;z)=\tf^{(0)}(w;z) \qquad (w \in S^{10}\text{ : word})
\end{equation*}
as the coefficient of $W=X_0W'X_1$. Thus the asymptotic condition \eqref{Asym_F} says that
\begin{align*}
f^{(0)}(w;z)=\tf^{(0)}(w;z) \to 0 \qquad (z \in \D_0 \to \infty)
\end{align*}
for any word $w$ in $S^{10}$ and the normalizing condition \eqref{Norm_F}
\begin{align*}
f^{(0)}(w;0)=\tf^{(0)}(w;0)=0
\end{align*}
for any word $w$ in $S^{10}$.\\

In the similar way, comparing the equation
\begin{align*}
F^{(1)}(z)&=\left(\sum_{i}\frac{f^{(1)}(t^{*}(\xi_0);z)^i}{i!}X_1^i\right)\left(\sum_{w \in S}f^{(1)}(\reg^{10}(t^{*}(w));z)W\right)\left(\sum_{j}\frac{f^{(1)}(t^{*}(\xi_1);z)^j}{j!}X_0^j\right)\\
&=\left(\sum_{i}\frac{(-f^{(1)}(\xi_1;z))^i}{i!}X_1^i\right)\left(\sum_{w \in S}f^{(1)}(\reg^{10}(t^{*}(w));z)W\right)\left(\sum_{j}\frac{(-f^{(1)}(\xi_0;z))^j}{j!}X_0^j\right)
\end{align*}
and \eqref{hF_reg1}, we have
\begin{align*}
f^{(1)}(\xi_0;z)&=\log(1-z),\\
f^{(1)}(\xi_1;z)&=-\log z,\\
f^{(1)}(w;z)&=\tf^{(1)}(t^{*}(w);z) \qquad (w \in S^{10}).
\end{align*}
Thus the Asymptotic condition \eqref{Asym_F} reads
\begin{equation*}
f^{(1)}(w;z)=\tf^{(1)}(t^{*}(w);z) \to 0 \qquad (z \in \D_1 \to \infty)
\end{equation*}
for any words $w$ in $S^{10}$.\\

Therefore the functions $f^{(i)}(w;z)$ satisfy the assumptions of Theorem \ref{thm:MPL_RH}, we have
\begin{align*}
f^{(0)}(w;z)&=\Li(w;z),\\
f^{(1)}(w;z)&=\Li(w;1-z).
\end{align*}

Therefore,
\begin{align}
F^{(0)}(z)&=\sum_{w \in S}\Li(w;z)W, \\
F^{(1)}(z)&=\sum_{w \in S}\Li(t^{*}(w);1-z)W=\sum_{w \in S}\Li(w;1-z)t_{*}(W)
\end{align}
are unique solutions to this Riemann-Hilbert problem. 
The last claim follows from \eqref{eq:fundamental0_0} and \eqref{eq:fundamental1_0}.
\end{proof}

The theorem says that the duality of the Drinfel'd associator \eqref{eq:dual_dr}, which is equivalent to the duality relation of multiple polylogarithms by Proposition \ref{prop:dual_dr}, can be interpreted as the consistency condition for this Riemann-Hilbert problem.


\end{document}